\documentclass[12pt]{amsart}
\usepackage[utf8]{inputenc}
\usepackage[english]{babel}
\usepackage{amsmath, amssymb, amsthm, amscd,color,comment}
\usepackage{cancel}
\usepackage{cite}
\usepackage{alltt}
\usepackage[dvipsnames]{xcolor}
\usepackage{array}
\usepackage[small,bf,labelsep=period]{caption}
\usepackage{mathtools}
\usepackage{hyperref}

\usepackage{enumerate}
\newtheorem{theorem}{Theorem}[section]

\newtheorem{example}[theorem]{Example}

\newtheorem{remark}[theorem]{Remark}
\newtheorem{lemma}[theorem]{Lemma}

\newtheorem{proposition}[theorem]{Proposition}

\DeclarePairedDelimiter\ceil{\lceil}{\rceil}
\DeclarePairedDelimiter\floor{\lfloor}{\rfloor}

\def\u{{\bf u}}
\def\v{{\bf v}}
\def\w{{\bf w}}
\def\x{{\bf x}}
\def\y{{\bf y}}

\def\N{\mathbb N}

\def\cL{\mathcal L}

\def\cP{\mathcal P}

\def\cX{\mathcal X}

\def\fqss{\mathbb F_{q^6}}

\def\fq{\mathbb F_q}

\def\supp{{\rm Supp}}

\def\Div{{\rm Div}}

\def\lub{{\rm lub}}
\def\glb{{\rm glb}}
\def\supp{{\rm Supp}}

\def\char{\mbox{\rm Char}}

\newcommand{\al}{\alpha}
\newcommand{\be}{\beta}

\begin{document}

\title[Complete set of Pure Gaps in Function Fields]{Complete set of Pure Gaps in Function Fields}



\thanks{The first author was partially supported by FAPEMIG. The second author was partially supported by FAPERJ/RJ-Brazil Grant 201.650/2021. The third author was partially funded by CNPq and FAPEMIG}

\author{Alonso S. Castellanos, Erik A. R. Mendoza, and Guilherme Tizziotti}

\address{Faculdade de Matemática, Universidade Federal de Uberlândia, Campus Santa Mônica, CEP 38400-902, Uberlândia, Brazil}
\email{alonso.castellanos@ufu.br}

\address{Instituto de Matemática, Universidade Federal do Rio de Janeiro, Cidade Universitária, CEP 21941-909, Rio de Janeiro, Brazil}
\email{erik@im.ufrj.br}

\address{Faculdade de Matemática, Universidade Federal de Uberlândia, Campus Santa Mônica, CEP 38400-902, Uberlândia, Brazil}
\email{guilhermect@ufu.br}

\begin{abstract}
In this work, we provide a way to completely determine the set of pure gaps $G_0(P_1, P_2)$ at two rational places $P_1, P_2$ in a function field $F$ over a finite field $\mathbb{F}_q$, and its cardinality. Furthermore, we given a bound for the cardinality of the set $G_0(P_1, P_2)$ which is better, in some cases, than the generic bound given by Homma and Kim in \cite{HK2001}. As a consequence, we completely determine the set of pure gaps and its cardinality for two families of function fields: the $GK$ function field and Kummer extensions. 
\end{abstract}

\maketitle

\section{Introduction}

The concept of \textit{pure gap} with respect to two rational places $P_1, P_2$ in a function field $F$ over a finite field $\mathbb{F}_q$ was introduced in 2001 by Homma and Kim \cite{HK2001}. In this work they described the set of pure gaps in terms of the gap sequences in the respective places and used this new notion to formulate lower bounds for the minimum distance of AG codes supported by two places. A few years later the study was extended to several points by Carvalho and Torres \cite{CT2005}. After these two papers, many authors sought to characterize Weierstrass semigroups and pure gaps in special families of curves, e.g. \cite{GM2001, GM2004, BMMQ2018, TT2019, BC2022, HY2018} and \cite{YH2018}. However, explicitly describing the set of pure gaps and determining its cardinality is complicated even for specific curves. This problem is challenging and important in its own right and can be related to several topics within the theory of curves over finite fields, such as limiting the number of rational points, e.g. \cite{BR2013, SV1986}. But the interest in the set of pure gaps is mainly due to its applications in the analysis of Algebraic Geometry codes (AG codes). In particular, the use of pure gaps is related to improvements on the Goppa bound of the minimum distance of these codes. In this sense, many works appeared with a study on these applications, e.g. \cite{CT2016, BQZ2018, HK2001, CMQ2016, CMQ2023}. 

In \cite{HK2001} Homma and Kim also introduced a special finite subset of the Weierstrass semigroup $H(P_1,P_2)$ at a pair of rational places $P_1, P_2$ carrying information about the whole semigroup, as well as information about the gaps and pure gaps. This set, denoted by $\Gamma(P_1, P_2)$, was called later by Matthews \cite{M2004}  the \emph{minimal generating set} of $H(P_1,P_2)$. In this work, using the minimal generating set $\Gamma(P_1, P_2)$, we provide a way to completely determine the set of pure gaps $G_0(P_1, P_2)$ at two rational places $P_1, P_2$ in a function field $F$ over a finite field $\mathbb{F}_q$, and its cardinality. Furthermore, we given a bound for the cardinality of the set $G_0(P_1, P_2)$ which is better, in some cases studied here, than the generic bound given by Homma and Kim in \cite{HK2001}. As a consequence, we completely determine the set of pure gaps and its cardinality for two families of function fields: the $GK$ function field and Kummer extensions. We observe that the cardinality of the pure gap set for some particular cases of Kummer extensions was determined in \cite{BQZ2018} and \cite{YH2018}, but our results allow us to obtain the same cardinalities for such cases in a simpler way. In addition, we show that certain families of curves reach the bounds for the cardinality of $G_0(P_1, P_2)$ given in this work, that is, we show that these bounds are sharp.

The organization of the paper is as follows. In Section 2, we establish the notations and present the preliminary results on the Weierstrass semigroup $H(P_1,P_2)$ and pure gap set $G_0(P_1,P_2)$ at two rational places. In Section 3, for rational places $P_1, P_2$ in an arbitrary function field $F$, we decompose the minimal generating set $\Gamma(P_1, P_2)$ using square regions with side $\pi$, where $\pi$ is the period of the semigroup $H(P_1, P_2)$. With this decomposition, we obtain a decomposition of the pure gap set $G_0(P_1, P_2)$ and, consequently, we obtain a complete description of $G_0(P_1, P_2)$ and a formula for its cardinality, see Proposition \ref{Puregapsymmetry}. In addition, we exhibit a bound for the cardinality of the pure gap set $G_0(P_1, P_2)$, see Proposition \ref{cota pure gaps}. In Section 4, we illustrate our results in two families of function fields: the $GK$ function field and Kummer extensions, see Theorem \ref{card G0} and Theorem \ref{Card of pure gap set of Kummer} respectively.

\section{Preliminaries and notation}

Let $q$ be the power of a prime $p$, $\fq$ be the finite field with $q$ elements, and $K$ be the algebraic closure of $\fq$. Given a function field of one variable $F/K$ of genus $g=g(F)$, we denote by $\mathcal P_F$ the set of places in $F$, by $\Div (F)$ the group of divisors in $F$, and for a function $z \in F$ we let $(z), (z)_\infty$ and $(z)_0$ stand for the principal, pole and zero divisor of the function $z$ in $F$ respectively. 

For a divisor $D\in \Div(F)$, we define the Riemann-Roch space associated to the divisor $D$ as 
$$
\cL(D)=\{z\in F: (z)+D\geq 0\}\cup \{0\},
$$
and we denote by $\ell(D)$ its dimension as vector space over $K$. Let $\mathbb{N}$ be the set of positive integers and $\mathbb{N}_0 = \mathbb{N} \cup \{0 \}$. For a place $P\in \cP_{F}$, the Weierstrass semigroup at $P$ is defined by
$$
H(P)=\{s\in \mathbb{N}_0: (z)_{\infty}=sP\text{ for some }z\in F\}.
$$
A non-negative integer $s$ is called a non-gap at $P$ if $s\in H(P)$. An element in the set $G(P):= \mathbb{N}_0 \setminus H(P)$ is called a gap at $P$ and, for a function field $F/K$ with genus $g>0$, $|G(P)|=g$.

The Weierstrass semigroup $H(P)$ defined using one place admits a generalization for two places. For $P_1$ and $P_2$ distinct places in $F$, we define the Weierstrass semigroup at $P_1, P_2$ by
$$
H(P_1, P_2)=\{(s_1, s_2)\in \N_0^2: (z)_{\infty}=s_1P_1+s_2P_2\text{ for some }z\in F\}.
$$
Analogously as in the case of one place, the elements of the set $G(P_1, P_2):=\N_0^2\setminus H(P_1, P_2)$ are called gaps at $P_1, P_2$ and can be characterized using Riemann-Roch spaces, that is, a pair $(s_1, s_2)\in \N_0^2$ is a gap at $P_1, P_2$ if and only if 
\begin{equation*}
\ell\left(s_1P_1+s_2P_2\right)=\ell\left(s_1P_1+s_2P_2-P_j\right)\text{ for some }j\in \{1, 2\}.
\end{equation*}
Moreover, the pair $(s_1, s_2)$ is called a pure gap at $P_1, P_2$ if it satisfies
$$
\ell\left(s_1P_1+s_2P_2\right)=\ell\left(s_1P_1+s_2P_2-P_j\right) \text{ for all }j\in \{1, 2\}
$$
or, equivalently,
$$
\ell\left(s_1P_1+s_2P_2\right)=\ell\left((s_1-1)P_1+(s_2-1)P_2\right).
$$
We denote the pure gap set at $P_1, P_2$ by $G_0(P_1, P_2)$. In \cite[Corollary 2.2]{HK2001}, Homma and Kim provided the following bound for the cardinality of $G_0(P_1, P_2)$:
\begin{equation}\label{cota Homma}
|G_0(P_1, P_2)| \leq \dfrac{g(g-1)}{2}.
\end{equation}
 
Now, let $G(P_1)$ and $G(P_2)$ be the gap sets of the places $P_1$ and $P_2$ respectively. For each $\be\in G(P_1)$, we let $\tau_{P_1, P_2}(\be):=\min\{\gamma \in \N_0 : (\be, \gamma)\in H(P_1, P_2)\}$. From \cite[Lemma 2.6]{K1994}, we have that $G(P_2) =\{\tau_{P_1, P_2}(\be) : \be\in G(P_1)\} $ and therefore 
\begin{equation} \label{tau}
\begin{array}{cccl}
\tau_{P_1, P_2}:& G(P_1) & \rightarrow & G(P_2)\\
&\be & \mapsto & \min\{\gamma \in \N_0 : (\be, \gamma)\in H(P_1, P_2)\}
\end{array}
\end{equation}
 is a bijective map. The graph of the map $\tau_{P_1, P_2}$ is denoted by $\Gamma (P_1, P_2)$, that is, 
$$
\Gamma (P_1, P_2):=\{(\be, \tau_{P_1, P_2}(\be)) : \be\in G(P_1)\}.
$$

For $\x=(\be_1, \gamma_1)$ and $\y=(\be_2, \gamma_2)$, the least upper bound and the greatest lower bound of $\x$ and $\y$ is defined as $\lub(\x, \y) =(\max\{\be_1, \be_2\}, \max\{\gamma_1, \gamma_2\})$ and $\glb(\x, \y) =(\min\{\be_1, \be_2\}, \min\{\gamma_1, \gamma_2\})$, respectively. Let $\preceq$ be the natural partial order in $\mathbb{N}_0^2$. We denote $\x\not\preceq\y$ if $\be_1 > \be_2$ or $\gamma_1 > \gamma_2$. 

The following results show that it is enough to determine $\Gamma(P_1, P_2)$ to compute the Weierstrass semigroup $H(P_1, P_2)$, for this reason the set $\Gamma(P_1, P_2)$ is called minimal generating set of the Weierstrass semigroup $H(P_1, P_2)$.

\begin{lemma}\cite[Lemma 2.2]{K1994}\label{genset}
Let $P_1$ and $P_2$ be two distinct places in $F$. Then 
$$
H(P_1, P_2)=\{\lub(\x, \y): \x, \y\in \Gamma(P_1, P_2)\cup (H(P_1)\times \{0\})\cup (\{0\}\times H(P_2))\}.
$$
\end{lemma}

The next result shows us that the set $\Gamma(P_1, P_2)$ is also important for determining the pure gap set $G_0(P_1, P_2)$.

\begin{lemma}\label{Gammasettopuregaps}
\cite[Proposition 5.1]{CMQ2023} Let $P_1$ and $P_2$ be two distinct places in $F$. Then
$$
G_0(P_1, P_2)=\{\glb(\x, \y): \x, \y\in \Gamma(P_1, P_2), \, \x\not\preceq\y, \, \y\not\preceq\x\}.
$$
\end{lemma}

\section{Pure gap set at two places}

Let $P_1$ and $P_2$ be distinct places in an arbitrary function field $F/K$ with genus $g$, and let $\pi$ be the period of the semigroup $H(P_1, P_2)$, where the period is defined by $\pi:=\min\{k\in \N : k(P_1-P_2) \text{ is a principal divisor}\}$. We begin this section by proving the following result that establishes a relationship between the period $\pi$ and the bijective map $\tau_{P_1, P_2}$ given in (\ref{tau}). 

\begin{proposition}\label{prop_period}
Let $P_1$ and $P_2$ be distinct places in $F$, and $\pi$ be the period of the semigroup $H(P_1, P_2)$. For $\be\in G(P_1)$ and $k\in\N$, the following statements are equivalent:
\begin{enumerate}[i)]
\item $\be+k\pi\in G(P_1)$.
\item $k\pi< \tau_{P_1, P_2}(\be)$.
\end{enumerate}
If any of the above conditions is satisfied, then
$$
\tau_{P_1, P_2}(\be+k\pi)=\tau_{P_1, P_2}(\be)-k\pi.
$$
\end{proposition}
\begin{proof}
Let $\beta \in G(P_1)$. For $k\in\N$ we know that there is a function $f$ such that $(f) = k \pi P_1 - k \pi P_2$.

$i) \Rightarrow ii)$. Suppose that $\be+k\pi\in G(P_1)$. If $k \pi \geq \tau_{P_1, P_2}(\be)$, since $(0,k \pi) \in H(P_1, P_2)$ and $(\beta , \tau_{P_1, P_2}(\be)) \in H(P_1, P_2)$, we have that $(\beta , k \pi) \in H(P_1, P_2)$. So, there is a function $h_1$ such that $(h_1) = D - (\beta P_1 + k \pi P_2)$, where $D$ is an effective divisor and $P_1, P_2 \notin \supp(D)$. Then, $(h_1 f^{-1}) = D - \beta P_1 - k \pi P_2 - k \pi P_1 + k \pi P_2 = D - (\beta + k \pi )P_1$, and thus we have $\beta + k \pi \in H(P_1)$, a contradiction. 

$ii) \Rightarrow i)$. Now, suppose that $k\pi< \tau_{P_1, P_2}(\be)$. If $\beta + k \pi \in H(P_1)$, then there is a function $h_2$ such that $(h_2) = D' + \alpha P_2 - (\beta + k \pi)P_1$, where $\alpha \in \N_0$ and $D'$ is an effective divisor with $P_1, P_2 \notin \supp(D')$. Thus, $(h_2f) = D' + \alpha P_2 - (\beta + k \pi)P_1 + k \pi P_1 - k \pi P_2 = D' - \beta P_1 - (k \pi - \alpha)P_2$. Note that,

$\diamond$ if $\alpha \geq k \pi$, then $\beta \in H(P_1)$, and we have a contradiction;

$\diamond$ if $0 \leq \alpha < k \pi$, then we have $0 < k \pi - \alpha \leq k \pi < \tau_{P_1, P_2}(\be)$ and $(\beta , k \pi - \alpha) \in H(P_1, P_2)$, a contradiction by the definition of $\tau_{P_1, P_2}(\be)$.

To obtain the equality in the last part, suppose that $\beta + k \pi \in G(P_1)$. Since $(\beta , \tau_{P_1, P_2}(\be)) \in H(P_1, P_2)$, $k \pi < \tau_{P_1, P_2}(\be)$ and $(f) = k \pi P_1 - k \pi P_2$, we can get  $(\beta + k \pi , \tau_{P_1, P_2}(\be) - k \pi) \in H(P_1 , P_2)$. So, $\tau_{P_1, P_2}(\be) - k \pi \geq \tau_{P_1, P_2}(\be+k \pi)$. On the other hand, since $(\be+k\pi , \tau_{P_1, P_2}(\be+k\pi))\in H(P_1, P_2)$ and $(f)=k\pi P_1-k\pi P_2$, then $(\be, \tau_{P_1, P_2}(\be+k\pi)+k\pi)\in H(P_1, P_2)$. By definition of $\tau_{P_1, P_2}(\be)$, we conclude that $\tau_{P_1, P_2}(\be)-k\pi \leq \tau_{P_1, P_2}(\be+k\pi)$ and the result follows.
\end{proof}

Let $\Gamma(P_1, P_2)$ be the minimal generating set of the semigroup $H(P_1, P_2)$. For $k_1, k_2\in \N_0$, define the set
\begin{align*}
\Gamma_{k_1, k_2}:&=\Gamma(P_1, P_2) \cap \left((k_1\pi, (k_1+1)\pi]\times (k_2\pi, (k_2+1)\pi]\right)\\
&=\{(\be, \tau_{P_1, P_2}(\be)): \be\in G(P_1), \, k_1\pi<\be<(k_1+1)\pi, \, k_2\pi < \tau_{P_1, P_2}(\be)< (k_2+1)\pi\}.
\end{align*}
Therefore, 
\begin{equation}\label{DecompGammaset}
\Gamma(P_1, P_2)=\bigcup_{\substack{0\leq k_1\\ 0\leq k_2}}\Gamma_{k_1, k_2}.
\end{equation}

\begin{remark} \label{empty}
Note that, if $(k_1, k_2) \neq (j_1,j_2)$, then $\Gamma_{k_1, k_2} \cap \Gamma_{j_1, j_2} = \emptyset$. So, the above union is disjoint. Furthermore, is well-known from Weierstrass semigroup's theory (see e.g. \cite[Lemma 2.1]{K1994} and Riemann-Roch's Theorem in \cite{S2009}) that $(a,b) \in H(P_1, P_2)$ if $a+b\geq 2g$. So, we have that if $k_1+k_2\geq \ceil*{\frac{2g-1}{\pi}}$, then $\Gamma_{k_1, k_2}=\emptyset$.
\end{remark}

Next, as a consequence of Proposition \ref{prop_period}, we show that to determine $\Gamma(P_1, P_2)$ it is enough to determine the sets of the form $\Gamma_{k, 0}$ for $k \geq 0$. 

For each $j\in \N_0$, define $$\w_{j}:=(-j\pi, j\pi) \in \mathbb{Z}^2.$$

\begin{proposition} 
For $i,j \geq 0$, 
\begin{equation*}\label{Gammasymmetry}
\Gamma_{i,j}=\Gamma_{i+j, 0}+\w_{j}.
\end{equation*}
In particular, 
\begin{equation}\label{Genus-Gamma}
\sum_{0\leq k}(k+1)|\Gamma_{k, 0}|=g.
\end{equation}
\end{proposition}
\begin{proof}
Let $(\be, \tau_{P_1, P_2}(\be))$ be an element in $\Gamma_{i+j, 0}$. Then $\be\in G(P_1)$, $(i+j)\pi< \be< (i+j+1)\pi$ and $0< \tau_{P_1, P_2}(\be)< \pi$. From definition of the period $\pi$ we have that $\pi\in H(P_1)$, therefore $\be-j\pi\in G(P_1)$ and $i\pi< \be-j\pi< (i+1)\pi$. Furthermore, from Proposition \ref{prop_period}, we obtain that $\tau_{P_1, P_2}(\be-j\pi)=\tau_{P_1, P_2}(\be)+j\pi$. Therefore $j\pi< \tau_{P_1, P_2}(\be-j\pi)< (j+1)\pi$ and
$$
(\be, \tau_{P_1, P_2}(\be))+\w_{j}=(\be-j\pi, \tau_{P_1, P_2}(\be-j\pi))\in \Gamma_{i, j}.
$$
This implies that $\Gamma_{i+j, 0}+\w_{j}\subseteq \Gamma_{i, j}$. The proof of the other inclusion is similar. To finish, using (\ref{DecompGammaset}) and the equality above, we have
$$
g=|\Gamma(P_1, P_2)|=\sum_{\substack{0\leq i \\0\leq j}}|\Gamma_{i, j}|=\sum_{\substack{0\leq i \\0\leq j}}|\Gamma_{i+j, 0}| =\sum_{0\leq k}(k+1)|\Gamma_{k, 0}|.
$$
\end{proof}

Analogously, for $k_1, k_2\in \N_0$ define the set  
$$
G_{k_1, k_2}:=G_0(P_1, P_2) \cap \left((k_1\pi, (k_1+1)\pi]\times (k_2\pi, (k_2+1)\pi]\right).
$$
Therefore, 
\begin{equation}\label{DecompPuregapset}
G_0(P_1, P_2)=\bigcup_{\substack{0\leq k_1\\ 0\leq k_2}}G_{k_1, k_2}.
\end{equation}

As in the case of $\Gamma_{k_1, k_2}$, we have that the above union is disjoint, and if $k_1 + k_2\geq \ceil*{\frac{2g-1}{\pi}}$, then $G_{k_1, k_2} = \emptyset$.

From Lemma \ref{Gammasettopuregaps}, we can rewrite the set $G_{k_1, k_2}$ as
\begin{equation}\label{Gk1k2}
G_{k_1, k_2}=\{\glb(\u, \v): \u\in \Gamma_{k_1, l_2}, \, \v \in \Gamma_{l_1 ,k_2}, \u\not\preceq\v, \, \v\not\preceq\u,  \, k_1\leq l_1, \, k_2\leq l_2\}.
\end{equation}

The following result shows that to determine $G_0(P_1, P_2)$ it is enough to determine the sets of the form  $G_{k, 0}$ for $k\geq 0$.
\begin{proposition}\label{Puregapsymmetry}
For $i,j \geq 0$,
$$
G_{i, j}=G_{i+j, 0}+\w_{j}. 
$$
In particular, 
\begin{equation}\label{CardinalityPuregaps}
|G_0(P_1, P_2)|= \sum_{ k = 0}^ {\ceil*{\frac{2g-1}{\pi}}-1}(k+1)|G_{k, 0}|.
\end{equation}
\end{proposition}
\begin{proof}
Let $\u\in G_{i+j, 0}$. From (\ref{Gk1k2}), there exist elements $\u_1\in \Gamma_{i+j, u}$ and $\u_2\in \Gamma_{v, 0}$, with $\u_1\not\preceq \u_2$ and $\u_2\not\preceq \u_1$, such that  
$$
\u=\glb(\u_1, \u_2),
$$
where $u\geq 0$ and $v\geq i+j$. From Proposition \ref{Gammasymmetry} we have 
$$
\v_1:=\u_1+\w_{j}\in \Gamma_{i, u+j}\quad \text{and}\quad\v_2:=\u_2+\w_{j}\in \Gamma_{v-j, j},
$$
where $u+j\geq j$ and $v-j\geq i\geq 0$. Furthermore, since $\v_1\not\preceq \v_2$ and $\v_2\not\preceq \v_1$ we obtain 
$$
\u+\w_{j}=\glb(\u_1, \u_2)+\w_{j}=\glb(\u_1+\w_{j}, \u_2+\w_{j})=\glb(\v_1, \v_2)\in G_{i, j}.
$$
So, $G_{i+j, 0}+\w_{j}\subseteq G_{i, j}$. The other inclusion is proved analogously. To finish the proof, from (\ref{DecompPuregapset}) and the equality above, we have
$$
|G_0(P_1, P_2)|=\sum_{\substack{0\leq i \\0\leq j}}|G_{i, j}|= \sum_{\substack{0\leq i \\0\leq j}}|G_{i+j, 0}| =\sum_{0\leq k}(k+1)|G_{k, 0}|.
$$
The result follows since $G_{k,0} = \emptyset$ for $k \geq \ceil*{\frac{2g-1}{\pi}}$.
\end{proof}

Note that, by (\ref{DecompPuregapset}) and Proposition \ref{Puregapsymmetry}, we have

\begin{equation}\label{conjunto pure gaps}
\displaystyle G_0(P_1, P_2) = \bigcup_{\substack{0\leq i,j \\ i+j < \ceil*{\frac{2g-1}{\pi}}}} \left( G_{i+j,0} + \w_{j}  \right) = \bigcup_{\substack{0 \leq j \leq k < \ceil*{\frac{2g-1}{\pi}}}} \left( G_{k,0} + \w_{j}  \right).
\end{equation}

Now, we will see how to determine the set $G_{k, 0}$ using the sets $\Gamma_{k_1, k_2}$. For this, fix $k\in \N_0$ and note that 
\begin{align*}
G_{k, 0}&=\{\glb(\u, \v): \u\in \Gamma_{k, k_2}, \, \v \in \Gamma_{k_1 ,0},\, \u\not\preceq \v, \, \v\not\preceq \u, \, k\leq k_1, \, 0\leq k_2\}\\
&=\{\glb(\u, \v): \u\in \Gamma_{k, k_2}, \, \v \in \Gamma_{k_1 ,0}, \,  \u\not\preceq \v, \, \v\not\preceq \u, \,  k< k_1, \, 0< k_2\}\\
&\quad \cup \{\glb(\u, \v): \u, \v\in \Gamma_{k, 0},  \, \u\not\preceq \v, \, \v\not\preceq \u\}\\
&\quad \cup\{\glb(\u, \v): \u\in \Gamma_{k, 0}, \, \v \in \Gamma_{k_1 ,0},\,\u\not\preceq \v, \, \v\not\preceq \u, \, k< k_1\}\\
&\quad \cup \{\glb(\u, \v): \u\in \Gamma_{k, k_2}, \, \v \in \Gamma_{k ,0}, \, \u\not\preceq \v, \, \v\not\preceq \u, \, 0< k_2\}\\
&=\{\glb(\u, \v): \u\in \Gamma_{k, k_2}, \, \v \in \Gamma_{k_1 ,0}, \, k< k_1, \, 0< k_2\}\\
&\quad \cup \{\glb(\u, \v): \u, \v\in \Gamma_{k, 0},  \, \u\not\preceq \v, \, \v\not\preceq \u\}\\
&\quad \cup\{\glb(\u, \v): \u\in \Gamma_{k, 0}, \, \v \in \Gamma_{k_1 ,0}, \, \u\not\preceq \v, \, k< k_1\}\\
&\quad \cup \{\glb(\u, \v): \u\in \Gamma_{k, k_2}, \, \v \in \Gamma_{k ,0}, \,\, \v\not\preceq \u, \, 0< k_2\}.\\
\end{align*}    
To simplify notation, denote
\begin{equation*}
\begin{array}{l}
G_{k, 0}^1=\{\glb(\u, \v): \u\in \Gamma_{k, k_2}, \, \v \in \Gamma_{k_1 ,0}, \, k< k_1, \, 0< k_2\},\\
G_{k, 0}^2=\{\glb(\u, \v): \u, \v\in \Gamma_{k, 0},  \, \u\not\preceq \v, \, \v\not\preceq \u\},\\
G_{k, 0}^3=\{\glb(\u, \v): \u\in \Gamma_{k, 0}, \, \v \in \Gamma_{k_1 ,0}, \, \u\not\preceq \v, \, k< k_1\}, \text{ and}\\
G_{k, 0}^4=\{\glb(\u, \v): \u\in \Gamma_{k, k_2}, \, \v \in \Gamma_{k ,0}, \,\, \v\not\preceq \u, \, 0< k_2\}.
\end{array}
\end{equation*}
Thus, for each $k\in \N_0$ we have that
\begin{equation}\label{G_k0}
G_{k, 0}=G_{k, 0}^1\cup G_{k, 0}^2\cup G_{k, 0}^3\cup G_{k, 0}^4.
\end{equation}
Note that, by definition, such union is disjoint and, by (\ref{conjunto pure gaps}), if we determine the sets $G_{k, 0}^i$, for $i=1,2,3,4$, then we can determine $G_0(P_1, P_2)$. Below we give some information about such sets.

\begin{lemma} \label{card G10}
$|G_{k, 0}^1| = \left(\sum_{k<k_1}|\Gamma_{k_1, 0}|\right)^2$.
\end{lemma}
\begin{proof}
Note that, for $\u_1, \u_2\in \Gamma_{k, k_2}$ and $\v_1, \v_2\in \Gamma_{k_1, 0}$, where $k<k_1$ and $0<k_2$, we have that $\glb(\u_1, \v_1)=\glb(\u_2, \v_2)$ if and only if $\u_1=\u_2$ and $\v_1=\v_2$. Therefore,
$$
|G_{k, 0}^1|=|\{\glb(\u, \v): \u\in \Gamma_{k, k_2}, \, \v \in \Gamma_{k_1 ,0}, \, k< k_1, \, 0< k_2\}|=\left(\sum_{k<k_1}|\Gamma_{k_1, 0}|\right)\left(\sum_{0<k_2}|\Gamma_{k, k_2}|\right).
$$
Moreover, from Proposition \ref{Gammasymmetry}, we have $\Gamma_{k, k_2}=\Gamma_{k+k_2, 0}+\w_{k_2}$. Thus, $|\Gamma_{k_, k_2}|=|\Gamma_{k+k_2, 0}|$ and 
\begin{equation*}
|G_{k, 0}^1|=\left(\sum_{k<k_1}|\Gamma_{k_1, 0}|\right)\left(\sum_{0<k_2}|\Gamma_{k, k_2}|\right)=\left(\sum_{k<k_1}|\Gamma_{k_1, 0}|\right)\left(\sum_{0<k_2}|\Gamma_{k+k_2, 0}|\right)=\left(\sum_{k<k_1}|\Gamma_{k_1, 0}|\right)^2.
\end{equation*}
\end{proof}

\begin{remark}\label{G2 G3 G4}
From definition we get 
\begin{equation}\label{boundsG^ik0}
\begin{array}{l}
0\leq |G_{k, 0}^2|\leq |\Gamma_{k, 0}|(|\Gamma_{k, 0}|-1)=|\Gamma_{k, 0}|^2-|\Gamma_{k, 0}|,\\
0\leq |G_{k, 0}^3|\leq |\Gamma_{k ,0}|\left(\sum_{k<k_1}|\Gamma_{k_1, 0}|\right), \text{ and}\\
0\leq |G_{k, 0}^4|\leq |\Gamma_{k ,0}|\left(\sum_{0<k_2}|\Gamma_{k, k_2}|\right)=|\Gamma_{k ,0}|\left(\sum_{k<k_1}|\Gamma_{k_1, 0}|\right).
\end{array}
\end{equation} 
\end{remark}

In the next result, we present bounds for the cardinality of $G_0(P_1, P_2)$.

\begin{proposition}\label{cota pure gaps} 
Let $P_1$ and $P_2$ be distinct places in the function field $F/K$ of genus $g$. Then
$$
\sum_{0\leq k}(k+1)\left(\sum_{k<k_1}|\Gamma_{k_1, 0}|\right)^2\leq |G_0(P_1, P_2)|\leq \sum_{0\leq k}(k+1)\left(\sum_{k\leq k_1}|\Gamma_{k_1, 0}|\right)^2-g.
$$
\end{proposition}

\begin{proof}
By Lemma \ref{card G10} and Remark \ref{G2 G3 G4} we have that
$$
|G_{k, 0}|\leq \left(\sum_{k<k_1}|\Gamma_{k_1, 0}|\right)^2+|\Gamma_{k, 0}|^2-|\Gamma_{k, 0}|+2|\Gamma_{k ,0}|\left(\sum_{k<k_1}|\Gamma_{k_1, 0}|\right)=\left(\sum_{k\leq k_1}|\Gamma_{k_1, 0}|\right)^2-|\Gamma_{k, 0}|.
$$
So, from Equation (\ref{CardinalityPuregaps}),
$$
\sum_{0\leq k}(k+1)\left(\sum_{k<k_1}|\Gamma_{k_1, 0}|\right)^2\leq |G_0(P_1, P_2)|
$$
and
\begin{align*}
|G_0(P_1, P_2)|&\leq \sum_{0\leq k}(k+1)\left[\left(\sum_{k\leq k_1}|\Gamma_{k_1, 0}|\right)^2-|\Gamma_{k, 0}|\right]\\
&=\sum_{0\leq k}(k+1)\left(\sum_{k\leq k_1}|\Gamma_{k_1, 0}|\right)^2-\sum_{0\leq k}(k+1)|\Gamma_{k, 0}|\\
&=\sum_{0\leq k}(k+1)\left(\sum_{k\leq k_1}|\Gamma_{k_1, 0}|\right)^2-g.& \text{(from Equation (\ref{Genus-Gamma}))}
\end{align*}
\end{proof}

Note that, we can rewrite the sets $G_{k, 0}^1$, $G_{k, 0}^2$, $G_{k, 0}^3$ and $G_{k, 0}^4$ in terms of the sets $\Gamma_{k_1, 0}$. In fact, 
\begin{equation}\label{G^i_k0}
\begin{array}{l}
G_{k, 0}^1=\{\glb(\u+\w_{k_2-k}, \v): \u\in \Gamma_{k_2, 0}, \, \v \in \Gamma_{k_1 ,0}, \, k< k_1, \, k < k_2\},\\
G_{k, 0}^2=\{\glb(\u, \v): \u, \v\in \Gamma_{k, 0},  \, \u\not\preceq \v, \, \v\not\preceq \u\},\\
G_{k, 0}^3=\{\glb(\u, \v): \u\in \Gamma_{k, 0}, \, \v \in \Gamma_{k_1 ,0}, \, \u\not\preceq \v, \, k< k_1\}, \text{ and}\\
G_{k, 0}^4=\{\glb(\u+\w_{k_2-k}, \v): \u\in \Gamma_{k_2, 0}, \, \v \in \Gamma_{k ,0}, \, \v\not\preceq \u+\w_{k_2-k}, \, k< k_2\}.
\end{array}
\end{equation}

\begin{remark} \label{Gi vazio}
By Remark \ref{empty}, if $k\geq \ceil*{\frac{2g-1}{\pi}}$ then $G_{k, 0}^i = \emptyset$ for each $i=1,2,3,4$. In addition, from definition, if $k= \ceil*{\frac{2g-1}{\pi}} - 1$ then $G_{k, 0}^1 = G_{k, 0}^3 = G_{k, 0}^4 = \emptyset$. 
\end{remark}

To finish this section we present the following result.

\begin{lemma}\label{lemma_diagonal} 
If $\be \equiv \tau_{P_1, P_2} (\be) \pmod{\pi}$ for each $\be \in G(P_1)$, then
$$
G_{k, 0}^2=\emptyset\quad\text{and}\quad G_{k, 0}^4=\{(b, a)\in \N^2: (a, b)\in G_{k, 0}^3\}-\w_k
$$  
for each $k\in \N_0$. In particular, $|G_{k, 0}^3| = |G_{k, 0}^4|$.
\end{lemma}
\begin{proof} Suppose that $\be \equiv \tau_{P_1, P_2} (\be) \pmod{\pi}$ for each $\be \in G(P_1)$. Let $\u = (\be_1, \tau_{P_1, P_2}(\be_1))$ and $\v = (\be_2, \tau_{P_1, P_2}(\be_2))$ in $\Gamma_{k, 0}$. So, $\be_1 = \tau_{P_1, P_2}(\be_1) + \ell_1 \pi$ and $\be_2 = \tau_{P_1, P_2}(\be_2) + \ell_2 \pi$ for some $\ell_1, \ell_2 \in \mathbb{Z}$. Since $k \pi < \be_1, \be_2 < (k+1) \pi$ and $0 < \tau_{P_1, P_2}(\be_1), \tau_{P_1, P_2}(\be_2) < \pi$, we have that $\ell_1 = \ell_2 = k$. Thus, if $\be_1 \leq \be_2$ (or $\be_2 \leq \be_1$), then $\tau_{P_1, P_2}(\be_1) \leq \tau_{P_1, P_2}(\be_2)$ (or $\tau_{P_1, P_2}(\be_2) \leq \tau_{P_1, P_2}(\be_1)$). Therefore, $\u\preceq \v$ or $\v\preceq \u$ for each $\u, \v\in \Gamma_{k, 0}$. This implies that $G_{k, 0}^2=\emptyset$.

On the other hand, let $(a, b)\in G_{k, 0}^3$. From (\ref{G^i_k0}), there exist elements $\u=(k\pi+j, j)\in \Gamma_{k, 0}$ and $\v=(k_1\pi+j_1, j_1)\in \Gamma_{k_1, 0}$, where $k<k_1$, such that $0<j_1<j<\pi$ and $(a, b)=\glb (\u, \v)=(k\pi+j, j_1)$. From Proposition \ref{Gammasymmetry}, $\v + \w_{k_1-k}\in \Gamma_{k, k_1-k}$. Moreover, $\u\npreceq \v+\w_{k_1-k}$. Thus, from (\ref{G^i_k0}) we obtain that 
$$
(b, a)-\w_k=(k\pi+j_1, j)=\glb(\u, \v+\w_{k_1-k})\in G_{k, 0}^4.
$$
This proves one inclusion. The other inclusion is analogous.
\end{proof}

\section{Pure gap set of some function fields}

In this section, as an application of the results introduced in the previous section, we completely determine the pure gap set and its cardinality of some families of function fields.

\subsection{The GK function field}
Consider the curve introduced by Giulietti and Korchm\'{a}ros \cite{GK2009} given by the equations
$$
\mathcal{GK}:\left\{
\begin{array}{l}
Z^{q^2-q+1}=Y\sum_{i=0}^{q} (-1)^{i+1} X^{i(q-1)}\\
Y^{q+1}=X^{q}+X
\end{array} \right. .
$$
This curve is maximal over $\fqss$ and is the first example of a maximal curve not covered by the Hermitian curve. Let $K$ be the algebraic closure of $\mathbb{F}_{q^6}$, and consider $GK=K(x,y,z)$ with $z^{q^2-q+1}=y\sum_{i=0}^{q} (-1)^{i+1} x^{i(q-1)}$ and $y^{q+1}=x^{q}+x$ the function field of $\mathcal{GK}$ over $K$. The $GK$ function field has $q^8 - q^6 + q^5 + 1$ rational places over $\fqss$, a single place at infinity denoted by $P_{\infty}$, and its genus is $g = \frac{1}{2} (q^3 + 1)(q^2 - 2) + 1$.

Let $GK(\mathbb{F}_{q^6})$ be the set of rational places in $GK$. For $j=0,\ldots,q-1$, let $P_j = (a_j,0,0)\in GK(\mathbb{F}_{q^6})$ be such that $a_{j}^{q}+a_j=0$, where $a_0 = 0$ and so $P_0=(0,0,0)$. For $\ell=1,\ldots,q^3 - q$, let $Q_{\ell}=(a_{\ell},b_{\ell},0)\in GK(\mathbb{F}_{q^6})$ be such that $b_{\ell}\neq 0$ and $a_{\ell}^{q}+a_{\ell}=b_{\ell}^{q+1}$. In the following, $P_{\infty}$, $P_j$, for $j=0,\ldots,q-1$, and $Q_{\ell}$, for $\ell=1,\ldots,q^3 - q$, will be the places given above. For each $j=0,\ldots,q-1$, we have the following principal divisor
\begin{equation} \label{divisors}
 (x-a_j)  =  (q^3 + 1)P_j - (q^3+1)P_{\infty}.
\end{equation}

\begin{theorem} \label{theorem bitransitive} \cite[Theorem 7]{GK2009}
The set $GK(\mathbb{F}_{q^6})$ of rational places of $GK$ splits into two orbits under the action of ${\rm Aut}(GK)$. One orbit, say $\mathcal{O}_1$, has size $q^3+1$ and consists of the places $P_j$ and $Q_{\ell}$ as above together with the place at infinity $P_{\infty}$. The other orbit has size $q^3(q^3+1)(q^2-1)$ and consists of the places $P=(a,b,c) \in GK(\mathbb{F}_{q^6})$ with $c \neq 0$. Furthermore, ${\rm Aut}(GK)$ acts on $\mathcal{O}_1$ as ${\rm PGU}(3,q)$ in its doubly transitive permutation representation.
\end{theorem}

\begin{proposition} \label{proposition H(P)} \cite[Proposition 1]{GK2009}
Let $P_j$ and $Q_{\ell}$ be as above. Then, $H(P_{\infty}) = H(P_j) = H(Q_{\ell}) = \langle q^3 - q^2 + q , q^3 , q^3 + 1 \rangle$, for each $j=1,\ldots, q-1$, and $\ell=1,\ldots,q^3-q$.
\end{proposition}

The following result provides an explicit description of the minimal generating set $\Gamma(P_0, P_\infty)$.  We will concentrate our results in the case $P_1=P_0$ and $P_2 = P_{\infty}$ but, by Theorem \ref{theorem bitransitive}, the results also continue to be valid for any places $P_{(a,b,0)} \in GK(\mathbb{F}_{q^6})$. That is, we can exchange the places $P_0$ and $P_{\infty}$ for any places on the orbit $\mathcal{O}_1$ given in Theorem \ref{theorem bitransitive}.

\begin{theorem}\cite[Theorem 3.4]{CT2016}\label{GammaGKcurve}
Let $q\geq 3$ and define the sets
\begin{align*}
&A_1=\{\gamma_{i, j, k} : 1\leq k \leq q-1, \, 0\leq i \leq k, \, k-i+1\leq j \leq q^2-q\},\\
&A_2=\{\gamma_{i, j, k} : 1\leq k \leq q-1, \, k+1\leq i \leq q, \, 0\leq j \leq q^2-q\},\\ 
&A_3=\{\gamma_{i, j, k} : q\leq k \leq q^2-q-2, \, 0\leq i \leq q, \, k-i+1\leq j \leq q^2-q\}, \text{ and}\\
&A_4=\{\gamma_{i, j, k} : q^2-q-1\leq k \leq q^2-1, \, k-q^2+q+1\leq i \leq q, \, k-i+1\leq j \leq q^2-q\}, 
\end{align*}
where 
$$\gamma_{i, j, k}=((k-1)(q^3+1)+(q+1-i)(q^2-q+1)-j, (i+j-k-1)(q^3+1)+(q+1-i)(q^2-q+1)-j).
$$
Then 
$$
\Gamma(P_0, P_{\infty})=A_1\cup A_2\cup A_3\cup A_4.
$$
\end{theorem}

In order to calculate the pure gap set $G_0(P_0, P_\infty)$, we are going to rewrite the set $\Gamma(P_0, P_\infty)$. With the notation introduced in Theorem \ref{GammaGKcurve}, note that
$$
A_1\cup A_2=\{\gamma_{i, j, k}: 1\leq k\leq q-1, \, 0\leq i \leq q, \, \max\{0, k-i+1\}\leq j \leq q^2-q \},
$$
$$
A_1\cup A_2 \cup A_3=\{\gamma_{i, j, k}: 1\leq k\leq q^2-q-2, \, 0\leq i \leq q, \, \max\{0, k-i+1\} \leq j \leq q^2-q \},
$$
and therefore
\begin{align}\label{Gamma(P0,Pinf) of GK}
\Gamma(P_0, P_\infty)=\{\gamma_{i, j, k}: & \, 1\leq k\leq q^2-1, \, \max\{0, k-q^2+q+1\}\leq i \leq q, \,\\
& \text{and } \max\{0, k-i+1\}\leq j \leq q^2-q\}.\nonumber
\end{align}
We observe that this new description of the minimal generating set $\Gamma (P_0, P_\infty)$ is also valid for the case $q=2$ since the result obtained coincides with the one given in \cite[Section III]{CT2016}.

\begin{proposition}\label{periodGKcurve}
The period of the semigroup $H(P_0, P_\infty)$ is $\pi = q^3+1$.
\end{proposition}
\begin{proof}
From (\ref{divisors}) we have $(x)=(q^3+1)(P_0-P_\infty)$, therefore $\pi \leq q^3+1$. Suppose that $1\leq \pi\leq q^3$ and let $w\in \fqss(GK)$ be the function such that $(w)=\pi(P_0-P_\infty)$. From Proposition \ref{proposition H(P)}, the smallest non-zero element of $H(P_\infty)$ is $q^3-q^2+q$, and since $\pi\in H(P_\infty)$ we obtain $q^3-q^2+q\leq \pi \leq q^3$. Furthermore, since
$$
(xw^{-1})=(q^3+1-\pi)(P_0-P_\infty),
$$ 
we conclude that $q^3+1-\pi\in H(P_\infty)$, where $0<q^3+1-\pi\leq q^2-q+1<q^3-q^2+q$, a contradiction. This implies that $\pi=q^3+1$.
\end{proof}


By the previous result and Remark \ref{empty}, we have that $\Gamma_{k, 0}= \emptyset$ for $k \geq q^2 -1$. Furthermore, from the definition of $\gamma_{i, j, k}$ and Proposition \ref{periodGKcurve}, we have that $\gamma_{i, j, k}\in \Gamma_{k-1, 0}$ if and only if $j=k-i+1$. Therefore, for $1\leq k \leq q^2-1$,
\begin{align*}
\Gamma_{k-1, 0}&=\{\gamma_{i, j, k}: \max\{0, k-q^2+q+1\}\leq i \leq q, \, \max\{0, k-i+1\}\leq j \leq q^2-q\}\\
&=\{\gamma_{i, j, k}: \max\{0, k-q^2+q+1\}\leq i \leq q, \, k-q^2+q+1\leq i \leq k+1\}\\
&=\{\gamma_{i, k-i+1, k}: \max\{0, k-q^2+q+1\}\leq i \leq \min\{q, k+1\}\}.
\end{align*}
Thus 
$$
\Gamma_{k, 0}=\{\gamma_{i, k-i+2, k+1}: \max\{0, k-q^2+q+2\}\leq i \leq \min\{q, k+2\}\}\text{ for }0\leq k \leq q^2-2.
$$

So, 
\begin{equation} \label{card gamma}
|\Gamma_{k, 0}|=\left\{\begin{array}{ll}
k+3, & \text{if }0\leq k \leq q-2,\\
q+1, & \text{if }q-1\leq k \leq q^2-q-3,\\
q^2-1-k, & \text{if }q^2-q-2\leq k \leq q^2-2.
\end{array}\right.
\end{equation}
Then, by Proposition \ref{cota pure gaps}, we have
$$
\begin{array}{lll}
|G_0(P_0, P_{\infty})| & \leq & \displaystyle \sum_{0\leq k}(k+1)\left(\sum_{k\leq k_1}|\Gamma_{k_1, 0}|\right)^2 - g \\
& = & \displaystyle \sum_{k=0}^{q-2} (k+1) \left( \sum_{k_1=k}^{q-2} (k_1 + 3) + \sum_{k_1=q-1}^{q^2-q-3} (q + 1) + \sum_{k_1=q^2-q-2}^{q^2-2} (q^2 - 1 - k_1) \right)^2\\
&  & + \displaystyle \sum_{k=q-1}^{q^2-q- 3} (k+1) \left(\sum_{k_1=k}^{q^2-q-3} (q + 1) + \sum_{k_1=q^2-q-2}^{q^2-2} (q^2 - 1 - k_1) \right)^2\\
&  & + \displaystyle \sum_{k=q^2-q- 2}^{q^2-2} (k+1) \left( \sum_{k_1=q^2-q-2}^{q^2-2} (q^2 - 1 - k_1)\right)^2 \\
& & \\
& & -  \left( \dfrac{1}{2} (q^3 + 1)(q^2 - 2) + 1 \right)\\
& & \\
& = &  \dfrac{10 q^{10} - 15q^8 - 4 q^7 + 20 q^6 - 56 q^5 - 35 q^4 + 124 q^3 - 40 q^2 - 4q}{120}.
\end{array}
$$
The last equality is obtained using the SageMath program \cite{SAGE}. We observe that this bound is better than the one given in (\ref{cota Homma}) for $q\geq 3$.

Now, using the description of the set $\Gamma_{k, 0}$, the relation (\ref{G_k0}) and Proposition \ref{Puregapsymmetry}, in the following lemmas we completely determine the sets $G_{k, 0}^1, G_{k, 0}^2, G_{k, 0}^3$ and $G_{k, 0}^4$ given in (\ref{G^i_k0}), and consequently we get $G_0(P_0, P_\infty)$. We observe that, by Remark \ref{Gi vazio}, if $k \geq q^2 - 2$ then $G_{k, 0}^1 = G_{k, 0}^3 = G_{k, 0}^4 = \emptyset$, and if $k \geq q^2 - 1$ then $G_{k, 0}^2 = \emptyset$.

\begin{lemma} \label{lemma 1}
For $0\leq k\leq q^2-3$ we have that
\begin{align*}
& G_{k, 0}^1=\left\{\gamma_{i_1, k_1-i_1+2, 1}^{i_2, k_2-i_2+2, k+1}:\begin{array}{l}
\max\{0, k_s-q^2+q+2\}\leq i_s \leq \min\{q, k_s+2\}\\
\text{and }k<k_s\leq q^2-2 \text{ for } s=1, 2 
\end{array}\right\} \text{ and }\\
&G_{k, 0}^3=\left\{\gamma_{i_1, k_1-i_1+2, 1}^{i_2, k-i_2+2, k+1}:\begin{array}{l}
\max\{0, k_1-q^2+q+2\}\leq i_1 \leq \min\{q, k_1+2\},\\
\max\{0, k-q^2+q+2\}\leq i_2 \leq \min\{q, k+2\},\\
i_2\leq i_1\text{ and }k<k_1\leq q^2-2 
\end{array}\right\},
\end{align*}
where 
$$
\gamma_{a_1, b_1, c_1}^{a_2, b_2, c_2}=((c_2-1)(q^3+1)+(q+1-a_2)(q^2-q+1)-b_2, (c_1-1)(q^3+1)+(q+1-a_1)(q^2-q+1)-b_1).
$$
\end{lemma}
\begin{proof}
We start by calculating the set $G_{k, 0}^1$. Let $k_1, k_2$ be positive integers such that $k<k_1, k_2 \leq q^2-2$. For $\gamma_{i_1, k_1-i_1+2, k_1+1}\in \Gamma_{k_1, 0}$ and $\gamma_{i_2, k_2-i_2+2, k_2+1}\in \Gamma_{k_2, 0}$ we have that
$$
\glb(\gamma_{i_2, k_2-i_2+2, k_2+1}+\w_{k_2-k}, \gamma_{i_1, k_1-i_1+2, k_1+1})=\gamma_{i_1, k_1-i_1+2, 1}^{i_2, k_2-i_2+2, k+1}.
$$
From (\ref{G^i_k0}), we conclude that
$$
G_{k, 0}^1=\left\{\gamma_{i_1, k_1-i_1+2, 1}^{i_2, k_2-i_2+2, k+1}:\begin{array}{l}
\max\{0, k_s-q^2+q+2\}\leq i_s \leq \min\{q, k_s+2\}\\
\text{and }k<k_s\leq q^2-2 \text{ for } s=1, 2 
\end{array}\right\}.
$$
Now, let $k_1$ be a positive integer such that $k<k_1\leq q^2-q$. For $\gamma_{i_2, k-i_2+2, k+1}\in \Gamma_{k, 0}$ and $\gamma_{i_1, k_1-i_1+2, k_1+1}\in \Gamma_{k_1, 0}$ we obtain that
\begin{align*}
\gamma_{i_2, k-i_2+2, k+1}\not\preceq \gamma_{i_1, k_1-i_1+2, k_1+1} &\Leftrightarrow (i_2-i_1)(q^2-q)<k_1-k\\
&\Leftrightarrow i_2\leq i_1.
\end{align*}
Assuming $i_2\leq i_1$, we deduce that
$$
\glb(\gamma_{i_2, k-i_2+2, k+1}, \gamma_{i_1, k_1-i_1+2, k_1+1})=\gamma_{i_1, k_1-i_1+2, 1}^{i_2, k-i_2+2, k+1}
$$
and therefore
$$
G_{k, 0}^3=\left\{\gamma_{i_1, k_1-i_1+2, 1}^{i_2, k-i_2+2, k+1}:\begin{array}{l}
\max\{0, k_1-q^2+q+2\}\leq i_1 \leq \min\{q, k_1+2\},\\
\max\{0, k-q^2+q+2\}\leq i_2 \leq \min\{q, k+2\},\\
i_2\leq i_1\text{ and }k<k_1\leq q^2-2 
\end{array}\right\}.
$$

\end{proof}

\begin{lemma} \label{lemma 2}
The following statements hold:
\begin{enumerate}[(i)]
\item  For $0\leq k\leq q^2-2$, $G_{k, 0}^2=\emptyset$.
\item For $0\leq k\leq q^2-3$, $G_{k, 0}^4=\{(b, a)\in \N^2: (a, b)\in G_{k, 0}^3\}-\w_k$, where $G_{k, 0}^3$ is given as in the previous lemma.
\end{enumerate}
\end{lemma}
\begin{proof}
From the description of $\Gamma (P_0, P_\infty)$ given in (\ref{Gamma(P0,Pinf) of GK}) and Proposition \ref{periodGKcurve} we have  
\begin{equation*}
\be \equiv \tau_{P_0, P_\infty} (\be) \pmod{q^3 + 1} \mbox{ for each } \be \in G(P_0).
\end{equation*}
The result follows directly from Lemma \ref{lemma_diagonal}.
\end{proof}

\begin{theorem} \label{card G0}
Let $G_{k,0}^{1}, G_{k,0}^{3}$ and $G_{k,0}^{4}$ be given in above lemmas. Then
$$
G_0(P_0, P_\infty) = \bigcup_{\substack{0 \leq j \leq k \leq q^2 -3}} \left( (G_{k,0}^{1} \cup G_{k,0}^{3} \cup G_{k,0}^{4}) + \w_{j}  \right)
$$
and
$$
|G_0(P_0, P_\infty)|=\frac{q(q-1)(10q^8+10q^7-25q^6-9q^5+71q^4-111q^3-86q^2+128q-12)}{120}.
$$
\end{theorem}
\begin{proof}
From (\ref{card gamma}) and Lemma \ref{card G10}, we get 
$$
|G_{k, 0}^1|=\left\{\begin{array}{ll}
\frac{1}{4}(2q^3 - k^2 - 7k - 10)^2, & \text{if }0\leq k \leq q-2,\\
\frac{1}{4}(2q^3 - 2(k + 4)q + q^2 - 2k + 3q - 4)^2, & \text{if }q-1\leq k \leq q^2-q-3,\\
\frac{1}{4}(q^4 - (2k + 3)q^2 + k^2 + 3k + 2)^2, & \text{if }q^2-q-2\leq k \leq q^2-3.
\end{array}\right.
$$
On the other hand, from the description of $G_{k, 0}^3$ given in Lemma \ref{lemma 1}, we can deduce that, for each $0\leq k \leq q^2-3$,
$$
|G_{k, 0}^3|=\sum_{i_2=\max\{0, k-q^2+q+2\}}^{\min\{q, k+2\}}\sum_{k_1=k+1}^{q^2-3}(\min\{q, k_1+2\}-\max\{i_2, k_1-q^2+q+2\}+1).
$$
So, to determine $|G_{k, 0}^3|$ we will separate the calculus in three cases.

$\diamond$ For $0\leq k \leq q-2$,
\begin{align*}
|G_{k, 0}^3|&=\sum_{i_2=0}^{k+2}\sum_{k_1=k+1}^{q^2-3}(\min\{q, k_1+2\}-\max\{i_2, k_1-q^2+q+2\}+1)\\
&=\sum_{i_2=0}^{k+2}\left(\sum_{k_1=k+1}^{q-2}(k_1+3-i_2)+\sum_{k_1=q-1}^{q^2-q-2+i_2}(q+1-i_2)+\sum_{k_1=q^2-q-1+i_2}^{q^2-3}(q^2-1-k_1)\right)\\
&=(k+3)\left(q^3-5-\frac{2k^2+39k+10}{12}-\frac{(2q^2-2q-5)(k+2)}{4}\right).
\end{align*}

$\diamond$ For $q-1\leq k \leq q^2-q-3$,
\begin{align*}
|G_{k, 0}^3|&=\sum_{i_2=0}^{q}\sum_{k_1=k+1}^{q^2-3}(\min\{q, k_1+2\}-\max\{i_2, k_1-q^2+q+2\}+1)\\
&=\sum_{i_2=0}^{q}\left(\sum_{k_1=k+1}^{q^2-q-2+i_2}(q+1-i_2)+\sum_{k_1=q^2-q-1+i_2}^{q^2-3}(q^2-1-k_1)\right)\\
&=(q+1)\left(\frac{3q^3+5q^2-8q-12}{6}-\frac{k(q+2)}{2}\right).
\end{align*}

$\diamond$ For $q^2-q-2\leq k \leq q^2-3$,
\begin{align*}
|G_{k, 0}^3|&=\sum_{i_2=k-q^2+q+2}^{q}\sum_{k_1=k+1}^{q^2-3}(\min\{q, k_1+2\}-\max\{i_2, k_1-q^2+q+2\}+1)\\
&=\sum_{i_2=k-q^2+q+2}^{q}\left(\sum_{k_1=k+1}^{q^2-q-2+i_2}(q+1-i_2)+\sum_{k_1=q^2-q-1+i_2}^{q^2-3}(q^2-1-k_1)\right)\\
&=\frac{q^2(q^4-3q^2+2)}{3}-\frac{k^2(k-3q^2+3)}{3}-\frac{k(3q^4-6q^2+2)}{3}.
\end{align*}

Now, by Lemma \ref{lemma 2}, $|G_{k, 0}^4| = |G_{k, 0}^3|$. Since $|G_{k, 0}|=|G_{k, 0}^1|+2|G_{k, 0}^3|$ for each $0\leq k\leq q^2-3$, then from (\ref{CardinalityPuregaps}) and after some calculations using SageMath program we get
$$
|G_0(P_0, P_\infty)|=\frac{q(q-1)(10q^8+10q^7-25q^6-9q^5+71q^4-111q^3-86q^2+128q-12)}{120}.
$$
\end{proof}

\begin{example}
For $q=2$, we have that  the pure gap set $G_0(P_0,P_\infty)$ of the $GK$ curve is given by 
$$
\begin{array}{llll}
G_0(P_0,P_\infty) & = & & \left( G_{0,0}^{1} \cup G_{0,0}^{3} \cup G_{0,0}^{4} \right) \\
                          &     & \bigcup & \left( G_{1,0}^{1} \cup G_{1,0}^{3} \cup G_{1,0}^{4} \right) \\
                          &  & \bigcup & \left( G_{1,0}^{1} \cup G_{1,0}^{3} \cup G_{1,0}^{4} \right) + (-9,9),
\end{array}
$$
where
$$
\begin{array}{l}
G_{0,0}^{1} =  \{ (1,1), (1,2), (1,4), (2,1), (2,2), (2,4), (4,1), (4,2), (4,4) \} \\
G_{0,0}^{3} =  \{ (3,1), (3,2), (5,1), (5,2), (5,4), (7,1), (7,2), (7,4) \} \\
G_{0,0}^{4} =  \{ (1,3), (1,5), (1,7), (2,3), (2,5), (2,7), (4,5), (4,7) \} \\
G_{1,0}^{1} =  \{ (10,1) \} \\
G_{1,0}^{3} =  \{ (11,1), (13,1)\}\\
G_{1,0}^{4} =  \{ (10,2), (10,4) \}.
\end{array}
$$

So, we get
$$G_0(P_0,P_\infty)=\left\{
\begin{array}{l} 
 (1, 1),
 (1, 2),
 (1, 3),
 (1, 4),
 (1, 5),
 (1, 7),
 (1, 10),
 (1, 11),
 (1, 13),
 (2, 1),\\
 (2, 2),
 (2, 3),
 (2, 4),
 (2, 5),
 (2, 7),
 (2, 10),
 (3, 1),
 (3, 2),
 (4, 1),
 (4, 2),\\
 (4, 4),
 (4, 5),
 (4, 7),
 (4, 10),
 (5, 1),
 (5, 2),
 (5, 4),
 (7, 1),
 (7, 2),
 (7, 4),\\
 (10, 1),
 (10, 2),
 (10, 4),
 (11, 1),
 (13, 1)
 \end{array}\right\},$$ 
and its cardinality is $|G_0(P_0,P_\infty)|=35$.
\end{example}

\subsection{Kummer extensions}
Let $K$ be the algebraic closure of $\mathbb{F}_q$ and consider the Kummer extension defined by the equation
\begin{equation*}\label{curveX}
\cX:\quad Y^m=f(X)^\lambda,
\end{equation*}
where $\lambda \in \mathbb{N}$, $m\geq 2$ is an integer such that $\char(K)=p\nmid m$, and $f(X)\in K[X]$ is a separable polynomial of degree $r\geq 2$ with $\gcd(m, \lambda r)=1$. Let $K(\mathcal{X})$ be its function field which has genus $g=\frac{1}{2} (m-1)(r-1)$. Denote by $P_\infty$ the single place at infinity and let $P_1, P_2$ be totally ramified places of the extension $K(\cX)/K(x)$ distinct to $P_\infty$. The pure gap set $G_0(P_1, P_2)$ was explicitly determined in \cite[Corollary 3]{HY2018}. In this subsection, we provide another explicit description of $G_0(P_1, P_2)$ and calculate its cardinality.

Let $\al_1$ and $\al_2$ be the roots of $f(X)$ corresponding to the places $P_1$ and $P_2$ respectively, then we have the principal divisor $(x-\al_i)=m(P_i-P_{\infty})$ for $i=1, 2$. So,

\begin{equation} \label{div x Kummer}
\left( \frac{x-\al_1}{x-\al_2} \right)=m(P_1-P_2).
\end{equation}

\begin{proposition}\label{periodKummer}
The period of the semigroup $H(P_1, P_2)$ is $\pi=m$.
\end{proposition}
\begin{proof}
By (\ref{div x Kummer}), we have that $\pi\leq m$. Suppose that $1\leq \pi \leq m-1$ and let $z\in K(\cX)$ be such that $(z)=\pi(P_1-P_2)$. From \cite[Proposition 3.3]{CMQ2023}, the smallest non-zero element of $H(P_2)$ is $m-\floor*{m/r}$, therefore $m-\floor*{m/r}\leq \pi \leq m-1$. Since  
$$
\left(\frac{x-\al_1}{z(x-\al_2)}\right)=(m-\pi)(P_1-P_2)\;,
$$ 
we obtain that $m-\pi\in H(P_2)$ and $0<m-\pi\leq m-(m-\floor*{m/r})=\floor*{m/r}<m-\floor*{m/r}$, a contradiction. This implies that $\pi=m$.
\end{proof}

By Remark \ref{empty}, we have that $\Gamma_{k, 0} = \emptyset$ for $k \geq r-1-\floor*{r/m}$. In order to determine $\Gamma_{k, 0}$ for $k \leq r-2-\floor*{r/m}$, we will use the following description of the minimal generating set $\Gamma(P_1, P_2)$ given in \cite[Theorem 8]{YH2017}:
\begin{align}\label{GammaKummer}
\Gamma(P_1, P_2)=\Bigg\{(mk_1+j, mk_2+j)\in \N^2 &: 1\leq j \leq m-1-\floor*{\frac{m}{r}},\, k_1\geq 0,\, k_2\geq 0, \\
& \quad \text{and } k_1+k_2=r-2-\floor*{\frac{rj}{m}}\Bigg\} \nonumber.
\end{align}
So, we get the following result.

\begin{proposition}\label{Gammak0-Kummer}
For $0\leq k \leq r-2-\floor*{r/m}$,
\begin{equation*}
\Gamma_{k, 0}=\left\{(mk+j, j)\in \N^2: \max\left\{1, m-\floor*{\frac{m(k+2)}{r}}\right\}\leq j \leq m-1-\floor*{\frac{m(k+1)}{r}}\right\}
\end{equation*}
and
\begin{equation*}
|\Gamma_{k, 0}|=\ceil*{\frac{m(k+2)}{r}}-\ceil*{\frac{m(k+1)}{r}}.
\end{equation*}
\end{proposition}

\begin{proof}
 From (\ref{GammaKummer}) and Proposition \ref{periodKummer}, we deduce that for $0\leq k\leq r-2-\floor*{r/m}$,
\begin{align*}
\Gamma_{k, 0}&=\left\{(mk+j, j): 1\leq j \leq m-1-\floor*{\frac{m}{r}}, \, k=r-2-\floor*{\frac{rj}{m}}\right\}\\
&=\left\{(mk+j, j): 1\leq j \leq m-1-\floor*{\frac{m}{r}}, \, m-\floor*{\frac{m(k+2)}{r}}\leq j\leq m-1-\floor*{\frac{m(k+1)}{r}}\right\}\\
&=\left\{(mk+j, j): \max\left\{1, m-\floor*{\frac{m(k+2)}{r}}\right\}\leq j \leq m-1-\floor*{\frac{m(k+1)}{r}}\right\}.
\end{align*}
This implies that
$$
|\Gamma_{k, 0}|=m-\floor*{\frac{m(k+1)}{r}}-\max\left\{1, m-\floor*{\frac{m(k+2)}{r}}\right\}.
$$
$\diamond$ If $k<r-2-\floor*{r/m}$ then $1< m-\floor*{m(k+2)/r}$ and  
$$
|\Gamma_{k, 0}|=\floor*{\frac{m(k+2)}{r}}-\floor*{\frac{m(k+1)}{r}}=\ceil*{\frac{m(k+2)}{r}}-\ceil*{\frac{m(k+1)}{r}}.
$$
$\diamond$ If $k=r-2-\floor*{r/m}$ then $m-\floor*{m(k+2)/r}\leq 1$ and 
$$
|\Gamma_{k, 0}|=m-\floor*{\frac{m(r-1)}{r}}-1=m-\ceil*{\frac{m(r-1)}{r}}.
$$
This concludes the proof.
\end{proof}

Now, using the before proposition, we will go to calculate the sets $G_{k,0}^i$ for $i=1,2,3,4$. We observe that, by Remark \ref{Gi vazio}, if $k \geq r-2-\floor*{r/m}$ then $G_{k, 0}^1 = G_{k, 0}^3 = G_{k, 0}^4 = \emptyset$, and if $k \geq r-1-\floor*{r/m}$ then $G_{k, 0}^2 = \emptyset$.

\begin{lemma}\label{lemma G^i_k,0 Kummer}
For $0\leq k \leq r-3-\floor*{r/m}$, 
\begin{align*}
& G_{k, 0}^1=\{(mk+j_2, j_1): 1\leq j_i\leq m-1-\floor*{m(k+2)/r}\text{ for }i=1, 2 \},\\
& G_{k, 0}^3=\left\{(mk+j, j_1):\begin{array}{l}
m-\floor*{m(k+2)/r}\leq j \leq m-1-\floor*{m(k+1)/r}\\
\text{and }1\leq j_1 \leq m-1-\floor*{m(k+2)/r} 
\end{array}\right\}, \text{ and }\\
& G_{k, 0}^4=\{(b, a)\in \N^2 : (a, b)\in G_{k, 0}^3\}-\w_k.
\end{align*}
Furthermore, $G_{k, 0}^2=\emptyset$ for $0\leq k \leq r-2-\floor*{r/m}$.
\end{lemma}
\begin{proof}
We start by noting that, from the description of $\Gamma(P_1, P_2)$ given in (\ref{GammaKummer}), $\be \equiv \tau_{P_1, P_2}(\be) \pmod{m}$ for each $\be\in G(P_1)$. Therefore, by Lemma \ref{lemma_diagonal}, we obtain the result for the sets $G_{k, 0}^2$ and $G_{k, 0}^4$.

For each $0\leq k\leq r-3-\floor*{r/m}$, let $\u=(mk_2+j_2, j_2)\in \Gamma_{k_2, 0}$ and $\v=(mk_1+j_1, j_1)\in \Gamma_{k_1, 0}$, where $k<k_i\leq r-2-\floor*{r/m}$ for $i=1, 2$. Then $\u+\w_{k_2-k}=(mk+j_2, m(k_2-k)+j_2)$ and therefore 
$$
\glb (\u+\w_{k_2-k}, \v)=(mk+j_2, j_1).
$$
Thus, from (\ref{G^i_k0}) and Proposition \ref{Gammak0-Kummer} we obtain that
$$
G_{k, 0}^1=\left\{(mk+j_2, j_1):\begin{array}{l}
\max\{1, m-\floor*{m(k_i+2)/r}\}\leq j_i \leq m-1-\floor*{m(k_i+1)/r}\\
\text{and }k<k_i\leq r-2-\floor*{r/m}\text{ for }i=1, 2 
\end{array}\right\}.\\
$$
Note that, for $i=1,2$, $\max\left\{1, m-\floor*{m(k_i+2)/r}\right\}=1$ if and only if $k_i=r-2-\floor*{r/m}$. So, considering the intervals in the set $G_{k, 0}^1$ above, for each $i=1,2$, we have that
\begin{align*}
& \text{for } k_i=k+1: \quad m-\floor*{\frac{m(k+3)}{r}}\leq j_i\leq m-1-\floor*{\frac{m(k+2)}{r}};\\
& \text{for } k_i=k+2: \quad m-\floor*{\frac{m(k+4)}{r}}\leq j_i\leq m-1-\floor*{\frac{m(k+3)}{r}};\\
& \hspace{12mm} \vdots \\
& \text{for } k_i=r-3-\floor*{r/m}: \quad m-\floor*{\frac{m(r-1-\floor*{r/m})}{r}}\leq j_i\leq m-1-\floor*{\frac{m(r-2-\floor*{r/m})}{r}};\\
& \text{for } k_i=r-2-\floor*{r/m}: \quad 1 \leq j_i\leq m-1-\floor*{\frac{m(r-1-\floor*{r/m})}{r}}.
\end{align*}
Thus, we have that $1\leq j_i\leq  m-1-\floor*{m(k+2)/r}$, and we get
$$
G_{k, 0}^1=\Bigg\{(mk+j_2, j_1): 1\leq j_i\leq m-1-\floor*{m(k+2)/r}\text{ for }i=1, 2 \Bigg\}.
$$

Now, in order to calculate the set $G_{k, 0}^3$, let $\u=(mk+j, j)\in \Gamma_{k, 0}$ and $\v=(mk_1+j_1, j_1)\in \Gamma_{k_1, 0}$, where $k<k_1\leq r-2-\floor*{r/m}$. Then, from Proposition \ref{Gammak0-Kummer}, 
$$
m-\floor*{\frac{m(k+2)}{r}}\leq j \leq m-1-\floor*{\frac{m(k+1)}{r}}
$$
and 
$$\max\left\{1, m-\floor*{\frac{m(k_1+2)}{r}}\right\}\leq j_1 \leq m-1-\floor*{\frac{m(k_1+1)}{r}}.
$$ 
Since $k<k_1$, we obtain that $mk+j<mk_1+j_1$ and 
$$
j_1\leq m-1-\floor*{\frac{m(k_1+1)}{r}}\leq m-1-\floor*{\frac{m(k+2)}{r}}\leq j-1<j.
$$
This implies that $\u \npreceq \v$, and from (\ref{G^i_k0}) we obtain that
\begin{align*}
G_{k, 0}^3&=\{\glb (\u, \v): \u\in \Gamma_{k, 0}, \, \v\in \Gamma_{k_1, 0}, \, k<k_1\}\\
&=\left\{(mk+j, j_1):\begin{array}{l}
m-\floor*{m(k+2)/r}\leq j \leq m-1-\floor*{m(k+1)/r},\\
\max\{1, m-\floor*{m(k_1+2)/r}\}\leq j_1 \leq m-1-\floor*{m(k_1+1)/r},\\
\text{and }k<k_1\leq r-2-\floor*{r/m} 
\end{array}\right\}.
\end{align*}

By similar arguments used to determine the set $G_{k, 0}^1$ above, we get
$$
G_{k, 0}^3=\left\{(mk+j, j_1):\begin{array}{l}
m-\floor*{m(k+2)/r}\leq j \leq m-1-\floor*{m(k+1)/r}\\
\text{and }1\leq j_1 \leq m-1-\floor*{m(k+2)/r}
\end{array}\right\}.
$$
\end{proof}

\begin{theorem}\label{Card of pure gap set of Kummer} 
Let $P_1$ and $P_2$ be totally ramified places of the extension $K(\cX)/K(x)$ distinct to $P_\infty$. Let $G^1_{k, 0}$, $G^3_{k, 0}$ and $G^4_{k, 0}$ be the sets given in the above lemma. Then
$$
G_0(P_1, P_2) = \bigcup_{\substack{0 \leq j \leq k \leq r-3-\floor*{r/m}}} \left( (G_{k,0}^{1} \cup G_{k,0}^{3} \cup G_{k,0}^{4}) + \w_{j}  \right)
$$
and
$$
|G_0(P_1, P_2)|=\sum_{k=1}^{r-2-\floor*{r/m}}k\left[\left(m-\ceil*{\frac{mk}{r}}\right)^2-\left(\ceil*{\frac{m(k+1)}{r}}-\ceil*{\frac{mk}{r}}\right)^2\right].
$$
\end{theorem}
\begin{proof}
From Lemma \ref{lemma G^i_k,0 Kummer}, we deduce that $|G_{k, 0}^2|=0$, $|G_{k, 0}^3|=|G_{k, 0}^4|$ and 
$$|G_{k, 0}^3|=|\{\glb (\u, \v): \u\in \Gamma_{k, 0}, \, \v\in \Gamma_{k_1, 0}, \, k<k_1\}|=|\Gamma_{k, 0}|\left(\sum_{k<k_1}|\Gamma_{k_1, 0}|\right)
$$ 
for each $0\leq k\leq r-3-\floor*{r/m}$. Note that in this case $|G_{k, 0}^3|$ and $|G_{k, 0}^4|$ reach the upper bound given in (\ref{boundsG^ik0}). Furthermore, from (\ref{G_k0}) and Lemma \ref{card G10}, we obtain that
$$
|G_{k, 0}|=\left(\sum_{k<k_1}|\Gamma_{k_1, 0}|\right)^2+2|\Gamma_{k, 0}|\left(\sum_{k<k_1}|\Gamma_{k_1, 0}|\right)=\left(\sum_{k\leq k_1}|\Gamma_{k_1, 0}|\right)^2-|\Gamma_{k, 0}|^2.
$$
Finally, from (\ref{CardinalityPuregaps}) and Proposition \ref{Gammak0-Kummer},
\begin{align*}
|G_0(P_1, P_2)|&=\sum_{0\leq k}(k+1)|G_{k, 0}|\\
&=\sum_{k=0}^{r-3-\floor*{r/m}}(k+1)\left[\left(\sum_{k_1=k}^{r-2-\floor*{r/m}}|\Gamma_{k_1, 0}|\right)^2-|\Gamma_{k, 0}|^2\right]\\
&=\sum_{k=0}^{r-3-\floor*{r/m}}(k+1)\left[\left(m-\ceil*{\frac{m(k+1)}{r}}\right)^2-\left(\ceil*{\frac{m(k+2)}{r}}-\ceil*{\frac{m(k+1)}{r}}\right)^2\right]\\
&=\sum_{k=1}^{r-2-\floor*{r/m}}k\left[\left(m-\ceil*{\frac{mk}{r}}\right)^2-\left(\ceil*{\frac{m(k+1)}{r}}-\ceil*{\frac{mk}{r}}\right)^2\right].
\end{align*}
\end{proof}

The cardinality of the pure gap set $G_0(P_1, P_2)$ is determined in \cite[Proof of Theorem 3.2]{BQZ2018} for the case $m=ur+1$, where $u$ is a positive integer, and in \cite[Corollary 12]{YH2018} for the case $m=(q+1)/N$, where $N$ is a divisor of $q+1$ such that $q-2-N\geq 0$ and $r=q$. Below, we verify that Theorem \ref{Card of pure gap set of Kummer} allows us to obtain the same cardinalities for such cases but in a simpler way when compared to the techniques used by the authors in the references cited above.\\

$\diamond$ Case $m=ur+1$ for some positive integer $u\in \N$. By Theorem \ref{Card of pure gap set of Kummer} we have that 
\begin{align*}
|G_0(P_1, P_2)|&=\sum_{k=1}^{r-2-\floor*{r/m}}k\left[\left(m-\ceil*{\frac{mk}{r}}\right)^2-\left(\ceil*{\frac{m(k+1)}{r}}-\ceil*{\frac{mk}{r}}\right)^2\right]\\
&=\sum_{k=1}^{r-2}k\left[\left(ur+1-\ceil*{\frac{(ur+1)k}{r}}\right)^2-\left(\ceil*{\frac{(ur+1)(k+1)}{r}}-\ceil*{\frac{(ur+1)k}{r}}\right)^2\right]\\
&=\sum_{k=1}^{r-2}k\left[\left(ur+1-(uk+1)\right)^2-\left((u(k+1)+1)-(uk+1)\right)^2\right]\\
&=u^2\left((r^2-1)\sum_{k=1}^{r-2}k-2r\sum_{k=1}^{r-2}k^2+\sum_{k=1}^{r-2}k^3\right)\\
&=u^2(r-1)(r-2)\left(\frac{r^2-1}{2}-\frac{r(2r-3)}{3}+\frac{(r-2)(r-1)}{4}\right)\\
&=\frac{u^2(r-1)(r-2)r(r+3)}{12}.
\end{align*}
Note that for $u=1$ we have 
$$
|G_0(P_1, P_2)|=\frac{(r-1)(r-2)r(r+3)}{12}.
$$
Now, from Proposition \ref{Gammak0-Kummer}, 
$$
|\Gamma_{k, 0}|=\ceil*{\frac{(r+1)(k+2)}{r}}-\ceil*{\frac{(r+1)(k+1)}{r}}=k+2-(k+1)=1,
$$
and thus
\begin{align*}
\sum_{k=0}^{r-2}(k+1)\left(\sum_{k_1=k}^{r-2}|\Gamma_{k_1, 0}|\right)^2-g &= \sum_{k=0}^{r-2}(k+1)\left(r-1-k\right)^2-\frac{r(r-1)}{2}\\
&= \sum_{k=1}^{r-1}k\left(r-k\right)^2-\frac{r(r-1)}{2}\\
&=r^2\sum_{k=1}^{r-1}k-2r\sum_{k=1}^{r-1}k^2+\sum_{k=1}^{r-1}k^3-\frac{r(r-1)}{2}\\
&=\frac{(r-1)r^3}{2}-\frac{(r-1)r^2(2r-1)}{3}+\frac{(r-1)^2r^2}{4}-\frac{r(r-1)}{2}\\
&=\frac{(r-1)(r-2)r(r+3)}{12}.
\end{align*}

Therefore, for $m=r+1$, the cardinality of $G_0(P_1, P_2)$ reaches the upper bound given in Proposition \ref{cota pure gaps}, and so we conclude that this upper bound is sharp.\\

$\diamond$ Case $m=(q+1)/N$, where $N$ is a divisor of $q+1$ such that $q-2-N\geq 0$ and $r=q$. So, $r-2-\lfloor r/m \rfloor=q-1-N=(m-2)N+N-2$. Writing $k=N\ell_k+j$, where $0\leq \ell_k \leq m-2$ and $0\leq j\leq N-1$, we have that 
$$\ceil*{ \frac{mk}{r}}=\ell_k+1,\mbox{ if $0\leq \ell_k \leq m-2$ and $0\leq j\leq N-1$}\,,$$ and $$\ceil*{\frac{m(k+1)}{r}}=\left\{ \begin{array}{ll}
\ell_k+1, & \mbox{if $0\leq \ell_k\leq m-2$ and $0\leq j\leq N-2,$}\\
\ell_k+2, & \mbox{if $0\leq \ell_k\leq m-3$ and $j=N-1$.}
\end{array} \right.$$

Now, for each $1\leq k\leq q-1-N$ define 
$$
h(k)=k\left[\left(m-\ceil*{\frac{mk}{r}}\right)^2-\left(\ceil*{\frac{m(k+1)}{r}}-\ceil*{\frac{mk}{r}}\right)^2\right].
$$
So, by Theorem \ref{Card of pure gap set of Kummer} we have that 
\begin{align*}
|G_0(P_1, P_2)|&=\sum_{k=1}^{r-2-\floor*{r/m}}h(k)\\
&=\sum_{j=1}^{N-1}h(j)+\sum_{\ell_k=1}^{m-3}\sum_{j=0}^{N-1}h(N\ell_k+j)+\sum_{j=0}^{N-2}h(N(m-2)+j)\\
&=\sum_{j=1}^{N-2}h(j)+h(N-1)+\sum_{\ell_k=1}^{m-3}\sum_{j=0}^{N-2}h(N\ell_k+j)+\sum_{\ell_k=1}^{m-3}h(N\ell_k+N-1)\\
&\quad +\sum_{j=0}^{N-2}h(N(m-2)+j)\\
&=\sum_{j=1}^{N-2}j(m-1)^2+(N-1)((m-1)^2-1)+\sum_{\ell_k=1}^{m-3}\sum_{j=0}^{N-2}(N\ell_k+j)(m-\ell_k-1)^2\\
&\quad +\sum_{\ell_k=1}^{m-3}(N(\ell_k+1)-1)[(m-\ell_k-1)^2-1]+\sum_{j=0}^{N-2}(N(m-2)+j)\\
&=\frac{(q+1)(m-1)}{12}((q+1)(m-1)-2m+N+7)-q(m-1).
\end{align*}

This formula to $|G_0(P_1, P_2)|$ is the same as the one given in \cite[Corollary 12]{YH2018}.
\bibliographystyle{abbrv}

\bibliography{puregaps} 

\end{document}